\documentclass{article}

\usepackage[
  left=1.25in,
  right=1.25in,
  top=1.2in,
  bottom=1.2in
]{geometry}

\usepackage{amsmath, amssymb, amsthm}
\usepackage{mathtools}
\usepackage{hyperref}
\usepackage{dsfont}

\usepackage{pgfplots}
\pgfplotsset{compat=1.18}

\usepackage[dvipsnames]{xcolor}
\usepackage{algorithm}
\usepackage{bbm}
\usepackage{enumitem}
\usepackage{algpseudocode}
\usepackage{bm}
\usepackage[section]{placeins}
\usepackage{tikz}
\usepackage{cancel}
\usepackage[style= numeric-comp,hyperref=true, doi=false,url=false,
            isbn=false,
            firstinits=true, sorting = none,
            block=none, backend=bibtex,maxnames=99]{biblatex}

\bibliography{generator}

\usetikzlibrary{arrows.meta,calc,positioning,fit,decorations.pathreplacing}

\usepackage{CJKutf8}

\newcommand{\cok}{\mathsf{co}^k}

\newcommand{\N}{\mathbb{N}}
\newcommand{\R}{\mathbb{R}}
\newcommand{\C}{\mathbb{C}}

\newcommand{\bfo}{\mathbf{1}}
\newcommand{\tr}{\operatorname{tr}}

\newcommand{\frakg}{\mathfrak{g}}
\newcommand{\frakh}{\mathfrak{h}}

\newtheorem{theorem}{Theorem}

\newtheorem{proposition}[theorem]{Proposition}

\newtheorem{lemma}{Lemma}

\newtheorem{mainthm}{Theorem}

\newtheorem{maincor}{Corollary}

\newtheorem{controlcor}{Corollary}

\newcommand{\rint}{\operatorname{int}}

\newcommand{\xc}[1]{\vspace{.1cm}
\noindent {\em #1} }

\title{Integrable Generators of Algebraic Vector Fields for \\ Real Connected Semi-simple Matrix Groups}

\author{Yiyang Jiang$^*$ \quad and \quad Xudong Chen\footnote{Y. Jiang and X. Chen are with the Department of Electrical \& Systems Engineering, Washington University, St. Louis, MO 63130, USA.  Emails: {\tt\small  \{j.yiyang, cxudong\}@wustl.edu}. Corresponding author: X. Chen.}}

\date{}

\begin{document}

\maketitle

\begin{abstract}
We show that for any real, connected,  semi-simple matrix group, its Lie algebra of algebraic vector fields can be generated by three complete vector fields. 
\end{abstract}

\section{Introduction and main results}\label{sec:introduction}

Let $G$ be a real, connected, semi-simple matrix group, by which we mean a connected, closed subgroup of the general linear group $\operatorname{GL}(\mathbb{V})$, for some real vector space $\mathbb{V}$, whose Lie algebra $\mathfrak{g}$ over~$\R$ is semi-simple 
(see, e.g.,~\cite[Section~VI]{Knapp} for classification of all simple real Lie algebras).  
The class of such groups includes all the identity components of the classical Lie groups given in~\cite[Section~I.17]{Knapp}. 
We also refer the reader to~\cite[Chapter 5]{OnishchikV} for a comprehensive study of these groups.  
We show in this paper that the Lie algebra of algebraic vector fields on any such group~$G$ can be generated by three complete vector fields.  
We formulate the result precisely.

Let $\Phi$ be the space spanned by the matrix coefficients associated with the standard representation of $G$ on~$\mathbb{V}$, and $\mathcal{A}(\Phi)$ be the unital algebra generated by $\Phi$. 
A vector field $V$ on $G$ is viewed as a map $V:G\to\operatorname{End}(\mathbb{V})$. 
We use $\mathfrak{X}(G)$ to denote the space of smooth vector fields on~$G$.  
The vector field $V$ is said to be \emph{algebraic} if, in a basis of $\mathbb{V}$, all matrix entries of $V$ belong to $\mathcal{A}(\Phi)$. Denote by $\mathfrak X_{\mathrm{alg}}(G)$ the space of algebraic vector fields. 
Further, we say that $V$ is {\it $\R$-complete} (or simply {\it complete}) if its flow $e^{tV}: G\to G$ is well defined for all $t\in \R$. 
Given a set $S\subseteq \mathfrak{X}(G)$, let $\operatorname{Lie}(S)$ be the Lie algebra generated by~$S$.

The main result of the paper is then the following:

\begin{mainthm}\label{thm:main}
For any real, connected, semi-simple matrix group $G$,   
there exist three complete algebraic vector fields $W_1$, $W_2$, and $W_3$ such that 
$\mathfrak X_{\mathrm{alg}}(G) = \operatorname{Lie}\{W_1,W_2,W_3\}$. 
\end{mainthm}

Throughout the paper, we use $\mathbb{N}$ to denote the set of positive integers and $\mathbb{N}_0 := \N \cup\{0\}$.  
Note that $\mathfrak{X}_{\mathrm{alg}}(G)$ is dense in $\mathfrak{X}(G)$ with respect to the compact-open $\mathrm{C}^k$-topology (or simply
$\cok$-topology) for any $k\in \N_0$ (see Lemma~\ref{lem:smooth-Vfieldapproximate} in Section~\ref{sec:proof-main}). 
Borrowing the terminology from~\cite{Varolin01,TV06}, we say that a real smooth manifold $M$ has the {\it density property} if every smooth vector field on $M$ can be approximated arbitrarily well in the $\cok$-sense, for any $k\in \N_0$, by Lie combinations of complete vector fields. Further, we say that $M$ has the {\it strong density property} if these complete vector fields  can be chosen out of some finite set. 
It then follows directly from Theorem~\ref{thm:main} that any real, connected, semi-simple matrix group $G$ has the strong density property.  

The problem of finding finitely many complete vector fields that generate the entire Lie algebra of algebraic vector fields has mainly been studied for {\it complex} manifolds (in the complex setting, completeness often refers to $\C$-completeness). 
For $\C^d$ with $d \geq 2$,  it has been shown in~\cite{AndristCn} that $\mathfrak{X}_{\mathrm{alg}}(\C^d)$ can be generated by three $\C$-complete vector fields (or simply completely $3$-generated). 
The result has been sharpened in~\cite{BeldievNew} that $\mathfrak{X}_{\mathrm{alg}}(\C^d)$ is completely $2$-generated.  
In~\cite{AndristSLtwo}, it has been shown that $\mathfrak{X}_{\mathrm{alg}}(\mathrm{SL}_2(\C))$ is completely $4$-generated and $\mathfrak{X}_{\mathrm{alg}}(M)$ is completely $5$-generated for $M$ the Danielewski surface $xy = z^2$. Further, in~\cite{And26}, it has been shown that for any Danielewski surface $M$ of the form $xy = p(z)$, where $p(z)$ is a polynomial with simple zeros, $\mathfrak{X}_{\mathrm{alg}}(M)$ is completely $6$-generated. 
Building upon the arguments of the current paper, we have shown in~\cite{CofDComplex} that $\mathfrak{X}_{\mathrm{alg}}(G)$, for any complex classical simple group $G$, is completely $3$-generated. The main result of~\cite{CofDComplex} thus provides an affirmative answer to the question posed in~\cite{AndristSLtwo}, which asked whether there exists a complex simple Lie group  $G$ other than $\mathrm{SL}_2(\C)$ such that $\mathfrak{X}_{\mathrm{alg}}(G)$ is finitely completely  generated.

We now describe a relevant consequence of Theorem~\ref{thm:main}.   
Let $\operatorname{Diff}(G)$ be the group of smooth diffeomorphisms of $G$ and $\operatorname{Diff}_0(G)$ be the connected component of $\operatorname{Diff}(G)$ that contains the identity $\operatorname{id}_G$.      
Let $W_1$, $W_2$, and $W_3$ be given as in the statement of Theorem~\ref{thm:main}, and $\mathcal{W}$ be the subgroup of $\operatorname{Diff}_0(G)$ generated by $e^{tW_i}$, for $t\in \R$ and $i = 1,2,3$.

The following result is a corollary of Theorem~\ref{thm:main}:

\begin{maincor}\label{cor:smoothdiffapproximate}
The subgroup $\mathcal{W}$ is dense in $\operatorname{Diff}_0(G)$ with respect to the $\cok$-topology, for any $k\in \N_0$. 
\end{maincor}

In~\cite{AC09}, it has been shown that if $V_1,\ldots, V_m$ are smooth vector fields on a compact manifold $M$ and if the group generated by $e^{tV_i}$, for $t\in \R$ and $i = 1,\ldots, m$, acts transitively on $M$, then for any $f\in\operatorname{Diff}_0(M)$, there exist an $N\in \N$, a sequence $i_1,\ldots, i_N\in \{1,\ldots, m\}$, and $a_1,\ldots, a_N\in \mathrm{C}^\infty(M)$ such that $f = e^{a_N V_{i_N}}\circ \cdots \circ e^{a_1 V_{i_1}}$. 
However, their result does not imply ours as the coefficients $a_1,\ldots, a_N$ in our setting can only be constants.

We further relate the above corollary to geometric control theory. 
Let $J^k(G,G)$ be the manifold of $k$-jets of smooth maps from $G$ to $G$. We denote by $j_x^kf$ the $k$-jet of $f\in \mathrm{C}^\infty(G, G)$ at $x\in G$. 
We embed $J^k(G,G)$ into a Euclidean space $\R^{N_k}$.  
For smooth $G$-valued maps $f,g$ defined on an open neighborhood of a compact set $K$ in $G$, let
$$
d_{k,K}(f,g):=\max_{x\in K}\left\|j_x^k f-j_x^k g\right\|,
$$
where $\|\cdot\|$ is the standard Euclidean norm. 
Note that the topology on $\operatorname{Diff}(G)$ induced by the above semimetrics for all compact sets $K$, coincides with the $\cok$-topology.    

Now, consider the following control-linear system on the group $G$: 
\begin{equation}\label{eq:controlsystem}
\dot x(t) = \sum_{i = 1}^3 u_i(t) W_i(x(t)),
\end{equation}
where each $u_i: [0,\infty)\to \R$, for $i = 1,2,3$, is a locally integrable function, viewed as a scalar control input. For ease of notation, let $u:=(u_1,u_2, u_3)$. 
System~\eqref{eq:controlsystem} is said to be {\it approximately controllable} on $\operatorname{Diff}_0(G)$ in the $\cok$-sense over the time interval $[0,T]$, for a given $T > 0$, if for any $f\in \operatorname{Diff}_0(G)$, any compact set $K\subseteq G$, and any error tolerance $\epsilon > 0$, there exists a $u\in \mathrm{L}^1([0,T],\R^3)$ such that the flow $g_t$ generated by the time-varying vector field $U(x,t):=\sum_{i = 1}^3 u_i(t) W_i(x)$ is well defined on an open neighborhood of $K$ in $G$ for all $t\in [0,T]$ and, moreover, $d_{k,K}(f, g_T) < \epsilon$. 
Since system~\eqref{eq:controlsystem} is control-linear,  approximate controllability over $[0,T]$ for some $T > 0$ implies approximate 
controllability over $[0,T]$ for {\it any} $T > 0$.

The following result is then an immediate consequence of Corollary~\ref{cor:smoothdiffapproximate}:

\begin{controlcor}
System~\eqref{eq:controlsystem} is approximately controllable on $\operatorname{Diff}_0(G)$ in the $\cok$-sense for all $k\in \N_0$. 
\end{controlcor}

For relevant work about controllability of diffeomorphisms, we mention~\cite{AS22} in which the authors have considered the class of control-linear systems on Riemannian manifolds~$M$. In particular, they have introduced the so-called ``strong Lie algebra approximating property'' (Assumption 3), which, roughly speaking, serves as the same purpose of Theorem~\ref{thm:main} (although the two conditions are significantly different).  Assuming that the control vector fields and their covariant derivatives of each other are uniformly bounded and that the aforementioned property is satisfied, the authors have shown that the control-linear system is approximately controllable on $\operatorname{Diff}_0(M)$ in the $\mathsf{co}^0$-sense.

The remainder of the paper is organized as follows: 
In Section~\ref{sec:three-generators}, we construct the three complete vector fields $W_1$, $W_2$, and $W_3$ and show that the space $\Phi\frakg_L + \frakg_L$ is contained in $\operatorname{Lie}\{W_1,W_2,W_3\}$, where $\frakg_L$ is the Lie algebra of left-invariant vector fields on~$G$. Then, in Section~\ref{sec:degree-growth}, 
we investigate the Lie algebra generated by $\Phi\frakg_L$ and show that the space of Lie products of depth $n$ is spanned by $\mathcal{A}_n(\Phi)\frakg_L$, where $\mathcal{A}_n(\Phi)$ is spanned by products of~$n$ arbitrary elements of~$\Phi$.  
Finally, in Section~\ref{sec:proof-main}, we prove Theorem~\ref{thm:main} and Corollary~\ref{cor:smoothdiffapproximate}.

\section{The three complete generators}
\label{sec:three-generators}

We identify $\frak{g}_L$ with $\frakg$, and denote by $\frakg_R$ the space of right-invariant vector fields. For a given $X\in \frakg$, let $L_X$ (resp., $R_X$) be the corresponding left-invariant (resp., right-invariant) vector field. 
Note that every vector field $V$ in $\frakg_L \oplus \frakg_R$ is algebraic and complete; indeed, for $V= L_X + R_Y$, we have that $e^{tV}(x) = e^{Yt}xe^{Xt}$ for any $x\in G$ and any $t\in \R$.  
Also, note that $[\frakg_L, \frakg_R] = 0$. 
Since $\frakg_L \oplus \frakg_R \approx \frakg \oplus \frakg$ is a semi-simple Lie algebra  (where $L_X + R_Y \mapsto (X,-Y)$ is a Lie algebra isomorphism), it is an immediate consequence of~\cite[Theorem~6]{Kuranishi1951} that there exist $W_1$ and $W_2\in \mathfrak g_L\oplus \mathfrak g_R$ such that $\operatorname{Lie}\{W_1,W_2\} = \mathfrak g_L\oplus\mathfrak g_R$. We construct explicitly these two vector fields.

Let $\frakg^\C$ be the complexification of $\frakg$, and  $\sigma:\frakg^\C\to\frakg^\C$ be the conjugation with respect to the real form~$\frakg$.
Let $\frakh$ be a Cartan subalgebra of $\frakg$ and  $\frakh^\C$ be its complexification.
Let $\Delta$ be the root system associated with $(\frakg^\C,\frakh^\C)$, and 
$\frakg^\C=\frakh^\C\oplus\bigoplus_{\alpha\in\Delta}(\frakg^\C)_\alpha$ be the root space decomposition of $\frakg^\C$. 
For each $\alpha\in\Delta$, let $\alpha^\sigma\in (\frakh^\C)^*$ be defined by $\alpha^\sigma(H):=\overline{\alpha(\sigma(H))}$ for any $H\in\frakh^\C$. 
It turns out that $\alpha^\sigma$ is a root, with its root space given by
$(\frakg^\C)_{\alpha^\sigma} = \sigma((\frakg^\C)_\alpha)$; indeed, since $\sigma$ is a conjugate-linear Lie algebra involution, 
$[H, \sigma(X)] = \overline{\alpha(\sigma(H))} \sigma(X)$ for any $H\in \frakh^\C$ and any $X\in \frakg_{\alpha}^\C$.

Let $c\in\R\setminus\{0,1,-1\}$, and $H\in\frakh$ be chosen such that 
$(\alpha - \beta)(H)\neq 0$ and $(\alpha - c\beta)(H) \neq 0$  for any two distinct roots $\alpha,\beta$. 
It follows that the numbers in the set $\{\alpha(H), c\alpha(H)\mid \alpha\in \Delta\}$ are pairwise distinct. 
Next, choose nonzero elements $E_\alpha\in \frakg^\C_\alpha$, for each $\alpha\in \Delta$, such that $\sigma(E_\alpha)=E_{\alpha^\sigma}$. 
We then define $E:=\sum_{\alpha\in\Delta} E_\alpha$. 
Since $\sigma(E_\alpha) = E_{\alpha^\sigma}$, we have that $\sigma(E)=E$ and hence, $E\in\frakg$.

With $c$, $H$, and $E$ given above, we now have the following lemma:

\begin{lemma}\label{lem:firsttwo}
Let $W_1:=L_H-R_{cH}$ and $W_2:=L_E-R_E$. 
Then, 
$\operatorname{Lie}\{W_1,W_2\}=\frakg_L\oplus\frakg_R$.
\end{lemma}

\begin{proof}
For convenience, let $\mathfrak{s}:=\operatorname{Lie}\{W_1,W_2\}$. 
It suffices to establish the result for the complexified version, i.e., we show that $\mathfrak{s}^\C = \frakg^\C_L \oplus \frakg^\C_R$.   
For any $\alpha\in\Delta$, 
$$
[W_1,L_{E_\alpha}]=\alpha(H)L_{E_\alpha} \quad \mbox{and} \quad 
[W_1, R_{E_\alpha}]=c\alpha(H) R_{E_\alpha}. 
$$ 
Thus, $L_{E_\alpha}$ and $R_{E_\alpha}$ are eigenvectors of $\operatorname{ad}_{W_1}:= [W_1,\cdot]$ with eigenvalues $\alpha(H)$ and $c\alpha(H)$, respectively. 
By repeatedly applying $\operatorname{ad}_{W_1}$ to $W_2$, we obtain that for any $m\in\N_0$, 
$$
\operatorname{ad}^m_{W_1}W_2= \sum_{\alpha\in\Delta}\alpha(H)^mL_{E_\alpha}-\sum_{\alpha\in\Delta}(c\alpha(H))^mR_{E_\alpha}.
$$
By our choice of $H$, the $2|\Delta|$ numbers $\alpha(H)$ and $c\alpha(H)$, for $\alpha\in\Delta$, are pairwise distinct. Thus, 
$W_2,\operatorname{ad}_{W_1}W_2,\ldots,\operatorname{ad}_{W_1}^{2|\Delta|-1}W_2$ are linearly independent, which in turn implies that 
$L_{E_\alpha},R_{E_\alpha}\in\mathfrak{s}^\C$ for every $\alpha\in\Delta$. Finally, note that $[L_{E_\alpha}, L_{E_{-\alpha}}]$, for $\alpha \in \Delta$, span $\{L_H \mid H\in \frakh^\C\}$ and the same applies for the right-invariant vector fields. We then conclude that $\mathfrak{s}^\C=\frakg_L^\C\oplus\frakg_R^\C$. 
\end{proof}

We next construct the third vector field~$W_3$. Decompose
$\mathfrak g=\oplus_{\ell = 1}^k \mathfrak g_\ell$, where each $\frakg_\ell$ is a simple ideal of~$\frakg$. For each $\ell$, let 
$\mathfrak g_\ell=\mathfrak k_\ell\oplus\mathfrak p_\ell$ be a Cartan decomposition, with $\mathfrak{k}_\ell$ a maximally compact subalgebra of $\frakg_\ell$.  
Let $Z_\ell\in \mathfrak{k}_\ell
$ be a nonzero element, and 
\begin{equation}\label{eq:defineZ}
Z:= \sum_{\ell = 1}^k Z_\ell.
\end{equation} 
Let $\chi$  be the character of the standard representation~$\mathbb{V}$, i.e., $\chi(x)$ is the trace of $x\in \operatorname{End}(\mathbb{V})$.   
We then define
\begin{equation}\label{eq:defineY3}
W_3(x):= \chi(x) L_Z(x). 
\end{equation}
It is clear that $\chi\in \Phi$, so $W_3\in \mathfrak{X}_{\mathrm{alg}}(G)$. 
We have the following fact:

\begin{lemma}
    The vector field $W_3$ is complete.
\end{lemma}

\begin{proof}
Let $G = K\exp{\mathfrak{p}}$ be the global Cartan decomposition of $G$, where $\mathfrak{p} := \oplus_{\ell = 1}^k \mathfrak{p}_\ell$ and $K$ is the compact subgroup of $G$ whose Lie algebra is $\mathfrak{k} := \oplus_{\ell = 1}^k \mathfrak{k}_\ell$. For any $x\in G$, assuming that $e^{s W_3}(x)$ exists for all $s\in [0,t]$, we have that 
$$e^{t W_3}(x) = x\exp\left (\int_0^t \chi(e^{s W_3}(x)) ds  Z \right ).$$
Since $Z\in \mathfrak{k}$, $e^{t W_3}(x)$ belongs to the compact set~$xK$. 
It follows from the Escape Lemma~\cite[Lemma~9.19]{GTM218} that the maximal integral curve $\gamma: t \mapsto e^{t W_3}(x)$ must be well defined for all $t\in \R$ and for all $x\in G$.  
\end{proof}

Let $\Phi\frakg_L$ be the space of vector fields spanned by $\phi L_X$ for $\phi\in \Phi$ and $X\in \frakg$. 
The main result of the section is the following theorem:

\begin{theorem}
\label{thm:degreeone}
It holds that 
$\Phi\mathfrak g_L\subseteq\operatorname{Lie}\{W_1,W_2,W_3\}.$
\end{theorem}

Let $\mathcal{A}(G)$ be the subspace of $\operatorname{End}(\mathbb{V})$ spanned by the elements of $G$. 
Note that $\mathcal{A}(G)$ is closed under multiplication; indeed, for any two elements $u := \sum_{i} a_i x_i$ and $v := \sum_{j} b_j y_j$, with $a_i, b_j\in \R$ and $x_i,y_j\in G$, 
$$
uv= \sum_{i,j} a_ib_j (x_iy_j) \in \mathcal{A}(G).
$$
It follows that $\mathcal{A}(G)$ is a unital subalgebra of $\operatorname{End}(\mathbb{V})$. 
Next, let $\mathcal{A}(\frakg)\subseteq \operatorname{End}(\mathbb{V})$ be the unital algebra generated by the elements of $\frakg$, known as the matrix representation of the universal enveloping algebra  of $\frakg$. We need the following lemma:

\begin{lemma}\label{lem:representation-algebra}
The two algebras $\mathcal{A}(G)$ and $\mathcal{A}(\frakg)$ are identical. 
\end{lemma}

\begin{proof}
For any $X\in\mathfrak g$, $\exp(tX)\in \mathcal{A}(G)$ for any $t\in \R$. 
Since $\mathcal{A}(G)$ is closed, we obtain that $$X = \lim_{t\to 0} \frac{1}{t}(\exp(tX) - I)\in \mathcal{A}(G).$$ Since $\mathcal{A}(G)$ is a unital algebra, $\mathcal{A}(\frakg)\subseteq \mathcal{A}(G)$.  
Conversely, $\mathcal{A}(\frakg)$ is a linear subspace of $\operatorname{End}(\mathbb{V})$ and hence, is closed. Thus, for every $X\in\mathfrak g$, the infinite series $\exp(X) = \sum_{k= 0}^\infty \frac{1}{k!} X^k$ is an element of~$\mathcal{A}(\frakg)$. 
The subgroup of $G$ generated by $\exp(X)$, for all $X\in \frakg$, contains an open neighborhood of the identity and hence, is an open subgroup of $G$. Since $G$ is connected, this subgroup coincides with~$G$.  It follows that $G\subseteq \mathcal{A}(\frakg)$. Since $\mathcal{A}(\frakg)$ is a vector space, we conclude that $\mathcal{A}(G)\subseteq \mathcal{A}(\frakg)$. 
\end{proof}

We next consider the trace map $\tr: \mathcal{A}(\frakg)\to \R$ and the associated bilinear trace form $\tr(AB)$ for $A, B\in \mathcal{A}(\frakg)$. The {\it kernel} of the trace form is $\{A \in \mathcal{A}(\frakg)\mid \tr(AB) = 0 \mbox{ for any } B\in \mathcal{A}(\frakg) \}$. The trace form is said to be {\it nondegenerate} if the kernel is trivial.    
We have the following result:

\begin{lemma}\label{lem:trace-pairing}
The trace form $\tr:\mathcal{A}(\frakg)\times \mathcal{A}(\frakg)\to\R$  is nondegenerate. 
\end{lemma}

\begin{proof}
We complexify $\mathcal{A}(\frakg)$ and extend the trace form bilinearly to $\mathcal{A}(\frakg)^\C\times \mathcal{A}(\frakg)^\C$. 
Decompose $\mathbb{V}^\C \approx \oplus_{\ell} m_\ell \mathbb{V}^\C_{\ell}$, where the $\mathbb{V}^\C_{\ell}$'s are pairwise nonisomorphic simple $\mathcal{A}(\frakg)^\C$-modules and the $m_\ell$'s are their multiplicities.  
  
It follows that for any $A\in \mathcal{A}(\frakg)^\C$,  
$\tr(A) = \sum_{\ell} m_\ell \tr(A |_{\mathbb{V}^\C_\ell})$. 
Let $\operatorname{Ann}(\mathbb{V}^\C_\ell)$ be the annihilator of $\mathbb{V}^\C_\ell$, i.e.,
$$
\operatorname{Ann}(\mathbb{V}^\C_\ell) := \left\{ A\in \mathcal{A}(\frakg)^{\C} \mid A v=0 \text{ for any } v\in \mathbb{V}^\C_\ell \right\}.
$$
Then, by~\cite[Lemma~3.2(b)]{BrownYakimov}, the kernel of the trace form coincides with $\cap_\ell\operatorname{Ann}(\mathbb{V}^\C_\ell) = \operatorname{Ann}(\mathbb{V}^\C)$. Finally, since 
$\mathcal{A}(\frakg)^\C$ is a subspace of $\operatorname{End}(\mathbb{V}^\C)$, we conclude that $\operatorname{Ann}(\mathbb{V}^\C) = \{0\}$. 
\end{proof}

For any $V\in \mathfrak{X}(G)$ and $\phi\in \mathrm{C}^\infty(G)$, 
let $V\phi\in \mathrm{C}^\infty(G)$ be the directional derivative of $\phi$ along $V$. 
Let $\mathcal{U}(\frakg_R)$ be the universal enveloping algebra of $\frakg_R$, and $\mathcal{U}(\frakg_R) \chi := \{U \chi \mid U \in \mathcal{U}(\frakg_R)\}$. 
It is clear that $\mathcal{U}(\frakg_R)\chi\subseteq \Phi$.     
We show below that $\chi$ is a {\it cyclic vector} of $\Phi$ as a $\mathcal{U}(\frakg_R)$-module.     

\begin{lemma}\label{lem:trace-coefficient-cyclic}
We have that $\mathcal{U}(\mathfrak g_R)\chi = \Phi$. 
\end{lemma}

\begin{proof}
For any $n\in \N$ and for any $X_1,\ldots, X_n\in \frakg$, 
$$(R_{X_n}\cdots R_{X_1}\chi)(x) = \tr(x X_1\cdots X_n).$$ 
It follows that
$$
\mathcal{U}(\frakg_R)\chi = \{ \phi_A (x):= \tr(xA) \mid A\in \mathcal{A}(\frakg)\}.
$$
Since $\mathcal{A}(G)$ is spanned by elements of $G$, 
$\phi_A$ vanishes on $G$ if and only if $\phi_A$ vanishes on $\mathcal{A}(G)$.  
By Lemma~\ref{lem:representation-algebra}, $\mathcal{A}(G) = \mathcal{A}(\frakg)$, so $\phi_A = 0$ if and only if $A$ is in the kernel of the trace form. Then, by Lemma~\ref{lem:trace-pairing}, $\phi_A = 0$ if and only if $A = 0$. 
Thus, as vector spaces, 
\begin{equation}\label{eq:uchia}
\mathcal{U}(\frakg_R)\chi \approx \mathcal{A}(\frakg) = \mathcal{A}(G) \approx \mathcal{A}(G)^*.
\end{equation} 
Note that each element $\phi\in \Phi$ can be naturally extended to a function $\hat \phi$ on $\mathcal{A}(G)$, simply given by $\hat \phi(\sum_{i}a_i x_i) := \sum_{i} a_i \phi(x_i)$. Let $\hat\Phi$ be the space spanned by all such elements $\hat \phi$, which is isomorphic to $\Phi$.  
Also, note that 
\begin{equation}\label{eq:hatphiinAG}
\mathcal{U}(\frakg_R)\chi\subseteq \Phi \approx \hat \Phi \subseteq \mathcal{A}(G)^*.
\end{equation} 
For~\eqref{eq:uchia} and~\eqref{eq:hatphiinAG} to hold, we must have that  $\mathcal{U}(\frakg_R)\chi =  \Phi$.    
\end{proof}

With the lemmas above, we now prove Theorem~\ref{thm:degreeone}. 

\begin{proof}[Proof of Theorem~\ref{thm:degreeone}]
Let $S:=\mathcal{U}(\frakg_L \oplus \frakg_R) W_3$ be the subspace of $\mathfrak{X}_{\mathrm{alg}}(G)$ spanned by $W_3$ and all the right-normed Lie products $[V_n, \cdots, [V_1, W_3]]$ for $V_i\in \frakg_L \oplus \frakg_R$.   
By Lemma~\ref{lem:firsttwo}, it suffices to show that $\Phi\frakg_L \subseteq S$. 
Consider the set 
$$\mathfrak{i} := \left\{X\in\frakg\mid \Phi L_X\subseteq S \right\}.$$ 
It is clear from its definition that $\mathfrak{i}$ is a linear subspace of $\frakg$. It is also nonempty because the element $Z$ defined in~\eqref{eq:defineZ} belongs to $\mathfrak{i}$; indeed, since $[\frakg_L, \frakg_R] = 0$ and by Lemma~\ref{lem:trace-coefficient-cyclic}, 
$$\mathcal{U}(\frakg_R)W_3 = (\mathcal{U}(\frakg_R)\chi ) L_Z = \Phi L_Z.$$
Furthermore, for any $X\in\mathfrak{i}$, any $Y\in\mathfrak{g}$, and any $\phi\in\Phi$, 
\begin{equation*}\label{eq:YbracketphiX}
[L_Y, \phi L_X]=(L_Y \phi)L_X + \phi L_{[Y,X]}.
\end{equation*}
Because $S$ is closed under $\frakg_L$,  $[L_Y, \phi L_X]\in S$. 
Also, since $L_Y\phi\in\Phi$ and $X\in\mathfrak{i}$, 
$(L_Y\phi)L_X\in  S$. Since $S$ is a vector space, we obtain that $\phi L_{[Y,X]}\in S$, which holds for any $\phi\in\Phi$. It follows that $[Y,X]\in\mathfrak{i}$ for any $Y\in \frakg$ and any $X\in \mathfrak{i}$, so $\mathfrak{i}$ is an ideal of $\mathfrak g$. 
Every ideal of $\mathfrak g=\oplus_{\ell = 1}^k \mathfrak{g}_\ell$ is a direct sum of a subfamily of the simple ideals $\mathfrak g_\ell$. By construction~\eqref{eq:defineZ}, each component $Z_\ell\in \frakg_\ell$ of $Z$ is nonzero. Since $Z\in\mathfrak{i}$, it holds that $\mathfrak{i} = \mathfrak{g}$. We conclude that $\Phi\mathfrak g_L \subseteq S\subseteq\operatorname{Lie}\{W_1,W_2,W_3\}$. 
\end{proof}

\section{Lie algebra generated by degree-one vector fields}\label{sec:degree-growth}
In this section, $G$ is an arbitrary real, connected, semi-simple Lie group (not necessarily a matrix group) and $\Phi$ is an arbitrary finite-dimensional subspace of $\mathrm{C}^\infty(G)$, closed under left-invariant vector fields.  
For each $n\in \N$, let $\mathcal{A}_n(\Phi)$ be the space spanned by $\phi_1\cdots \phi_n$ with $\phi_i\in \Phi$, and $\mathcal{A}_0(\Phi) := \R\bfo$, where $\bfo$ is the constant function $\bfo(x) = 1$ for all $x\in G$.  
Let $\mathcal{L}_n$ be the space of right-normed Lie products $[V_1,\cdots,[V_{n-1}, V_n]]$ of depth~$n$ with $V_i\in \Phi \frakg_L$. Note that by the Jacobi identity, any Lie product is a linear combination of right-normed Lie products. 

The main result of the section is the following:

\begin{theorem}
\label{thm:liealgebra}
If $\Phi$ does not contain any 
nonzero constant function, then 
$
\mathcal{L}_n=\mathcal{A}_n(\Phi)\mathfrak g_L 
$ for any $n \in \N$. 
\end{theorem}

It suffices to establish the result for the complexified version, i.e., $\mathcal{L}^\C_n = \mathcal{A}_n(\Phi)^\C \frakg_L^\C$. For ease of notation, we will omit their super-index for the remainder of the section. All spaces below are over~$\C$ unless otherwise stated. 
Also, we will use $X\in \frakg$ as a shorthand notation for the left-invariant vector field $L_X$.

The proof will be carried out by induction on $n$. For the base case $n = 1$, the result holds trivially. For the inductive step, we need to show that
\begin{equation*}
[\Phi\frakg_L, \mathcal{A}_{n-1}(\Phi) \frakg_L] = \mathcal{A}_n(\Phi) \frakg_L.  
\end{equation*}

First, note that for any $\phi\in \Phi$, any $\zeta := \phi_1\cdots \phi_{n-1}\in \mathcal{A}_{n-1}(\Phi)$, and any $X, Y\in \frakg_L$,    
\begin{equation}\label{eq:liebracket}
[\phi X, \zeta Y ] = \phi\zeta [X, Y] + (X\zeta) \phi Y - (Y\phi) \zeta X. 
\end{equation}
Since  $\mathcal{A}_n(\Phi)$ is closed under $\frakg_L$ for all $n \in\N$, the right hand side of~\eqref{eq:liebracket} belongs to $\mathcal{A}_n(\Phi)\frakg_L$, so $[\Phi\frakg_L, \mathcal{A}_{n-1}(\Phi) \frakg_L] \subseteq \mathcal{A}_n(\Phi) \frakg_L$.
The main effort of the proof is in establishing the reverse inclusion:
\begin{equation}\label{eq:converseinclusion}
[\Phi\frakg_L, \mathcal{A}_{n-1}(\Phi) \frakg_L] \supseteq \mathcal{A}_n(\Phi) \frakg_L.
\end{equation}

We recall that $\mathfrak{h}$ is a Cartan subalgebra of $\frakg$, $\Delta$ is the set of roots associated with $(\frakg, \frakh)$, and that $\frakg = \frakh \oplus \bigoplus_{\alpha \in \Delta} \frakg_\alpha$ is the root space decomposition. 
Let $B(X, Y):= \operatorname{tr}(\operatorname{ad}_X\operatorname{ad}_Y)$ be the Killing form. 
For each $\alpha\in \Delta$, let $h_\alpha\in \frakh$ be defined such that $\alpha(H) = B(h_\alpha, H)$ for any $H\in \frakh$.  
Let $\langle \alpha, \beta \rangle:= B(h_\alpha, h_\beta)$, which is an inner-product on the $\R$-span of $\Delta$. 
We extend the inner-product to~$\frakh^*$. 
For each nonzero $\alpha\in \frakh^*$, we follow the notation of~\cite{Knapp} and write  $$(\beta,\alpha):= 2\langle \beta, \alpha\rangle/\|\alpha\|^2,$$
for all $\beta\in \frakh^*$.  
Further, let $H_\alpha \in \frakh$ be defined such that $\beta(H_{\alpha}) = (\beta, \alpha)$.  
In particular, $\alpha(H_\alpha) = 2$.

As a consequence of the Chevalley's basis, there exist elements $X_\alpha\in \frakg_\alpha$, for $\alpha\in \Delta$, such that the following commutation relations hold: 
$$
\begin{aligned}
\relax 
[X_\alpha, X_{-\alpha}] & = H_\alpha, \quad \mbox{for any } \alpha \in \Delta, \\
[H_\alpha, X_\beta] & = (\beta,\alpha) X_\beta,  \quad \mbox{for any } \alpha, \beta\in \Delta,  \\
[X_\alpha, X_\beta] & = c_{\alpha, \beta} X_{\alpha + \beta}, \quad \mbox{for nonproportional } \alpha \mbox{ and } \beta, 
\end{aligned}
$$
where $c_{\alpha,\beta}^2 = (r + 1)^2$, with $r$ the largest integer such that $(\beta - r\alpha)$ is a root (by convention, $X_{\alpha + \beta} = 0$ if $\alpha + \beta \notin\Delta \cup \{0\}$).

Next, let $\Lambda$ be the set of weights associated with the representation $\Phi$ of $\frakg$, and $
\Phi = \oplus_{\lambda \in\Lambda}\Phi_\lambda$ 
be the weight space decomposition. By the highest weight theorem, each $\lambda\in \Lambda$ belongs to the $\R$-span of $\Delta$, so $(\lambda, \alpha)\in \R$ for all $\alpha\in \Delta$.  
It is clear that $\mathcal{A}_n(\Phi)\frakg_L$ is spanned by elements $\xi H_{\alpha}, \xi X_{\alpha}$, where 
 $\alpha\in \Delta$ and $$\xi = \prod_{i = 1}^n\phi_i, \quad \mbox{ with } \phi_i\in \Phi_{\lambda_i} \mbox{ for some } \lambda_i\in \Lambda.$$    
For any such $\xi$, we show below that $\xi H_\alpha, \xi X_\alpha\in [\Phi\frakg_L, \mathcal{A}_{n-1}(\Phi)\frakg_L]$. For convenience, for each $i = 1,\ldots, n$, let 
$\xi_{-i}:= \prod_{j = 1, j\neq i}^n \phi_j$. Also, let 
\begin{equation}\label{eq:defsigmaanddelta}
\lambda:= \frac{1}{2}\sum_{i = 1}^{n} \lambda_i.
\end{equation} 

We start with the following proposition: 

\begin{proposition}\label{prop:case1}
    If there exists some $i$ such that $\lambda_i\neq \lambda$, then $\xi H\in \mathcal{L}_n $ for any $H\in \frakh$.
\end{proposition}

\begin{proof}
Let $\delta_i:= 2(\lambda - \lambda_i) \neq 0$. 
It directly follows from computation that
$$
[\phi_i H, \xi_{-i} H]  = (H \xi_{-i}) \phi_i H - (H\phi_i) \xi_{-i} H 
= \delta_i(H) \xi H.
$$
If $\delta_i(H)\neq 0$, then $\xi H\in \mathcal{L}_n$.  
Otherwise, let $H':= H + H_{\delta_i}$. Since $\delta_i(H') = \delta_i(H_{\delta_i}) = 2$, both $\xi H'$ and $\xi H_{\delta_i}$ belong to $\mathcal{L}_n$ and hence,   
$\xi H = \xi H' -  \xi H_{\delta_i} \in \mathcal{L}_n$.
\end{proof}

We now assume that $\lambda_i = \lambda$ for all $i = 1,\ldots, n$ and show that $\xi H\in \mathcal{L}_n$. It follows from~\eqref{eq:defsigmaanddelta} that $(n - 2)\lambda = 0$. 
Consider the following two cases:
\begin{description}
    \item[\it Case 1:] $n = 2$ and $\lambda \neq 0$;
    \item[\it Case 2:] $\lambda = 0$.
\end{description}

We first deal with Case~1. For any given $\alpha\in \Delta$, let $\mathfrak{s}_\alpha$ be the Lie sub-algebra of $\frakg$ spanned by $\{H_\alpha, X_\alpha, X_{-\alpha}\}$, which is isomorphic to $\mathfrak{sl}_2(\C)$. We need the following lemma: 

\begin{lemma}\label{lem:sl2rep}
For any $\phi \in \Phi_\lambda$ and for any $\alpha\in \Delta$ with $(\lambda, \alpha) \geq 0$, there exists a decomposition  $\phi = \phi' + \phi''$, with $\phi',\phi''\in \Phi_\lambda$, such that the following hold
\begin{enumerate}
\item There exists an element $\psi\in \Phi_{\lambda - \alpha}$ such that $X_\alpha  \psi = \phi'$. 
\item  $\mathfrak{s}_\alpha \phi'' = 0$.
\end{enumerate}
\end{lemma}

\begin{proof}
Let $\Phi'$ be the smallest subspace of $\Phi$ that contains $\phi$ and is closed under $\mathfrak{s}_\alpha$, so $\Phi'$ is a representation of $\mathfrak{s}_\alpha$. 
Decompose $\Phi'= \oplus_{\ell}\Phi'_{\ell}$ such that each $ \Phi'_{\ell}$ is irreducible.  
Correspondingly, we write $\phi = \sum_{\ell} \phi_{\ell}$ where  $\phi_{\ell}\in \Phi'_{\ell}$. 
Note that each $\phi_\ell$  necessarily belongs to $\Phi_\lambda$. 
Let  $m_\ell\in \N_0$ be the highest weight of~$\Phi'_{\ell}$, and $k:=(\lambda, \alpha) \geq 0$. 
Then, $m_\ell \geq k$ and, moreover, $m_\ell$ and $k$ share the same parity. 
Let 
$\phi':= \sum_{\ell: m_\ell > 0} \phi_\ell$ and 
$\phi'':= \sum_{\ell: m_\ell = 0} \phi_\ell$. 
By construction, $\mathfrak{s}_\alpha \phi'' = 0$.  
For $\phi'$, note that $${X_\alpha}{X_{-\alpha}}\phi_{\ell} = \frac{1}{4}( m_\ell - k + 2)(m_\ell + k)\phi_{\ell}.$$  
Thus, by setting $$\psi:= \sum_{\ell:m_\ell > 0} \frac{4}{(m_\ell - k + 2)(m_\ell + k)} {X_{-\alpha}} \phi_{\ell},$$
we obtain ${X_\alpha}\psi = \phi'$.   
\end{proof}

With the lemma above, we now establish the following proposition: 

\begin{proposition}\label{prop:case2}
    Let $\phi_1, \phi_2\in \Phi_\lambda$, with $\lambda\neq 0$. 
    Then, $\phi_1\phi_2 H_\alpha \in \mathcal{L}_2$ for any $\alpha\in \Delta$.
\end{proposition}

\begin{proof}
We consider the following two sub-cases: 

\xc{Sub-case 1.1: $(\lambda, \alpha) \neq 0$.} Without loss of generality, we assume that $(\lambda, \alpha) > 0$. Then, $\mu:=\lambda - \alpha$ is a weight. 
We appeal to Lemma~\ref{lem:sl2rep} to obtain the decomposition $\phi_1 = \phi_1' + \phi_1''$. Note that $H_\alpha \phi_1'' = (\lambda, \alpha) \phi_1''$. 
Since $(\lambda, \alpha)\neq 0$, we must have that $\phi_1'' = 0$ and hence, $\phi_1' = \phi_1$.  
Let $\psi_1\in \Phi_\mu$ be such that $X_\alpha \psi_1 = \phi_1$. Then, by~\eqref{eq:liebracket}, 
\begin{equation}\label{eq:case2-1}
\begin{aligned}\relax
[\phi_2 H_\alpha, \psi_1  X_\alpha] & = \psi_1\phi_2[H_\alpha, X_\alpha] + (\mu,\alpha)\psi_1 \phi_2 X_\alpha - \psi_1 (X_\alpha \phi_2) H_\alpha \\
& = (\lambda, \alpha) \psi_1 \phi_2 X_\alpha - \psi_1 (X_\alpha \phi_2) H_\alpha \in \mathcal{L}_2.
\end{aligned}
\end{equation}
Note that $X_\alpha \phi_2\in \Phi_\nu$, where $\nu:= \lambda + \alpha$. 
Since $\nu - \mu = 2\alpha \neq 0$, it follows from Proposition~\ref{prop:case1} that $\psi_1 \psi_2 H_\alpha\in \mathcal{L}_2$ for any $\psi_2\in \Phi_\nu$. Using~\eqref{eq:case2-1} and the fact that $(\lambda, \alpha)\neq 0$, we have that $\psi_1 \phi_2 X_\alpha \in \mathcal{L}_2$. Using~\eqref{eq:liebracket} again, we obtain that  
\begin{equation*}\label{eq:case2-2}
\begin{aligned}\relax
[\psi_1 H_\alpha, \phi_2  X_\alpha] & = \psi_1\phi_2[H_\alpha, X_\alpha] + (\lambda, \alpha) \psi_1\phi_2 X_\alpha - (X_\alpha \psi_1) \phi_2 H_\alpha \\
& = (\nu, \alpha) \psi_1\phi_2 X_\alpha - \phi_1\phi_2 H_\alpha \in \mathcal{L}_2.
\end{aligned}
\end{equation*}
Since $\psi_1\phi_2 X_\alpha\in \mathcal{L}_2$, we have that $\phi_1\phi_2 H_\alpha\in \mathcal{L}_2$.

\xc{Sub-case 1.2: $(\lambda, \alpha) = 0$.} 
We use Lemma~\ref{lem:sl2rep} to decompose $\phi_i = \phi'_i + \phi''_i$, for $i= 1,2$. In case $\phi'_i$ is nonzero, we let $\psi_i\in \Phi_{\lambda- \alpha}$ be such that $X_\alpha\psi_i = \phi'_i$. 
We then write 
\begin{equation}\label{eq:case2-3}
\phi_1 \phi_2 H_\alpha = \phi'_1\phi'_2 H_\alpha + \phi'_1\phi''_2 H_\alpha + \phi''_1 \phi'_2 H_\alpha + \phi''_1\phi''_2 H_\alpha.
\end{equation}
We show below that each term on the right hand side of~\eqref{eq:case2-3} belongs to $\mathcal{L}_2$. 
For the first term $\phi'_1 \phi'_2 H_\alpha$, we let $H':= H_\alpha + H_\lambda$ and note that $\lambda(H')- \alpha(H') = 0$.  
Then, 
$$
\begin{aligned}\relax
[\phi'_2 H', \psi_1 X_\alpha ] & = \psi_1 \phi'_2 \alpha(H') X_\alpha + (\lambda - \alpha)(H') \phi'_2\psi_1 X_\alpha - \psi_1 (X_\alpha \phi'_2) H' \\
& = 2 \psi_1 \phi'_2 X_\alpha - \psi_1 (X_\alpha \phi'_2) H' \in \mathcal{L}_2, \\
[\psi_1 H_\alpha, \phi'_2 X_\alpha] & = 2\psi_1 \phi'_2 X_\alpha - \phi'_1\phi'_2 H_\alpha \in \mathcal{L}_2. 
\end{aligned}
$$
Using the same arguments as in Sub-case~1.1, $\psi_1 (X_\alpha \phi'_2) H'\in \mathcal{L}_2$ which implies that $\psi_1\phi'_2 X_\alpha\in \mathcal{L}_2$ and hence, $\phi'_1\phi'_2 H_\alpha\in \mathcal{L}_2$. 
For the second term $\phi'_1\phi''_2 H_\alpha$ of~\eqref{eq:case2-3}, we similarly have that
$$
\begin{aligned}\relax
[\phi''_2 H', \psi_1 X_\alpha ] & = 2\psi_1 \phi''_2 X_\alpha \in \mathcal{L}_2,  \\
[\psi_1 H_\alpha, \phi''_2 X_\alpha] & = 2\psi_1 \phi''_2 X_\alpha - \phi'_1\phi''_2 H_\alpha \in \mathcal{L}_2, 
\end{aligned}
$$
so $\phi'_1\phi''_2 H_\alpha\in \mathcal{L}_2$. By the same arguments (and by symmetry), the third term $\phi''_1\phi'_2 H_\alpha$ also belongs to $\mathcal{L}_2$.
Finally, for the last term $\phi''_1\phi''_2 H_\alpha$, we have that
$$
[\phi''_1 X_\alpha, \phi''_2 X_{-\alpha}] = \phi''_1\phi''_2 H_\alpha \in \mathcal{L}_2.
$$
This completes the proof.
\end{proof}

We now deal with Case 2, i.e., $\lambda_i = 0$ for all $i = 1,\ldots, n$. 
Choose a positive system and let $\Delta^+$ be the set of positive roots. 
By the hypothesis of Theorem~\ref{thm:liealgebra}, $\Phi$ does not contain any nonzero constant function, so $\Phi$ does not contain a trivial representation of $\frakg$.  
It follows from the highest weight theorem that the weight space $\Phi_0$ can be expressed as 
$\Phi_0 = \sum_{\alpha\in \Delta^+} X_{\alpha}\Phi_{-\alpha}$. 
By multi-linearity, we can assume, without loss of generality, 
that $\phi_1,\ldots, \phi_n$ satisfy the condition that there exist $\alpha_i\in \Delta^+$ and $\psi_i \in \Phi_{-\alpha_i}$, for $i = 1,\ldots, n$, such that
\begin{equation}\label{eq:xalphapsi}
X_{\alpha_i}\psi_i  = \phi_i. 
\end{equation}
Let $\Phi_i$ be the smallest subspace of $\Phi$ that contains $\psi_i$ and is closed under $\mathfrak{s}_{\alpha_i}$. 
Decompose $\Phi_i = \oplus_\ell\Phi_{i,\ell}$, where each $\Phi_{i,\ell}$ is irreducible under $\mathfrak{s}_{\alpha_i}$. 
Let $m_{i,\ell}$ be the highest weight of $\Phi_{i,\ell}$. 
Since $0$ is a weight of $\Phi_{i,\ell}$, $m_{i,\ell}$ is an even integer. Moreover, it follows from~\eqref{eq:xalphapsi} that $m_{i,\ell}$ must be strictly positive because otherwise no such $\psi_i$ would exist for the relation to hold.  
For ease of presentation but without loss of generality, we can assume that the representation $\Phi_i$  is itself irreducible. 

We now establish the following result: 

\begin{proposition}\label{prop:case3}
    Suppose that  each $\phi_i$ is a zero-weight vector of the irreducible representation~$\Phi_i$ of $\mathfrak{s}_{\alpha_i}$ with highest weight $m_i > 0$; then, $\xi H\in \mathcal{L}_n$ for all $H\in \mathfrak{h}$.     
\end{proposition}

\begin{proof}
Let $\theta_i,\psi_i\in \Phi_{i}$ be such that $X_{-\alpha_i} \theta_i = X_{\alpha_i}\psi_i = \phi_i$.
We first show that $\xi H_{\alpha_i} \in \mathcal{L}_n$ for all $i = 1,\ldots, n$.  
We consider two sub-cases:

\xc{Sub-case 2.1: There exist $i,j\in \{1,\ldots, n\}$ such that $\alpha_i \neq \alpha_j$.}  By computation, 
\begin{equation*}\label{eq:case21}
\begin{aligned}\relax
[\psi_i H_{\alpha_i}, \xi_{-i} X_{\alpha_i}] & = 2\psi_i \xi_{-i} X_{\alpha_i} - \xi H_{\alpha_i}, \\
[\psi_j H_{\alpha_i}, \xi_{-j} X_{\alpha_j}] & = (\alpha_j, \alpha_i) \psi_j\xi_{-j} X_{\alpha_j} - \xi H_{\alpha_i}, \\
[\psi_i H_{\alpha_j}, \xi_{-i} X_{\alpha_i}] & = (\alpha_i,\alpha_j) \psi_i \xi_{-i} X_{\alpha_i} - \xi H_{\alpha_j}, \\
[\psi_j H_{\alpha_j}, \xi_{-j} X_{\alpha_j}] & = 2\psi_j\xi_{-j} X_{\alpha_j} - \xi H_{\alpha_j}. 
\end{aligned}
\end{equation*}
We then obtain that
\begin{multline*}\label{eq:case3-2}
(\alpha_i, \alpha_j)(\alpha_j, \alpha_i) [\psi_i H_{\alpha_i}, \xi_{-i} X_{\alpha_i}] - 4 [\psi_j H_{\alpha_i}, \xi_{-j} X_{\alpha_j}] - 2(\alpha_j,\alpha_i) [\psi_i H_{\alpha_j}, \xi_{-i} X_{\alpha_i}] \\ + 2(\alpha_j,\alpha_i)[\psi_j H_{\alpha_j}, \xi_{-j} X_{\alpha_j}] = (4 - (\alpha_i, \alpha_j)(\alpha_j, \alpha_i)) \xi H_{\alpha_i} \in \mathcal{L}_n.
\end{multline*}
Since $\alpha_i$ and $\alpha_j$ are distinct positive roots, $4 - (\alpha_i, \alpha_j)(\alpha_j, \alpha_i)  > 0$ and hence, $\xi H_{\alpha_i}\in \mathcal{L}_n$.

\xc{Sub-case 2.2: All the $\alpha_i$'s are the same.} Let $\alpha:= \alpha_i$ be the common root. 
Note that 
 $$X_\alpha \phi_i = \frac{1}{4}m_i (m_i + 2) \theta_i \quad \mbox{and} \quad 
 X_{-\alpha} \phi_i = \frac{1}{4}m_i (m_i + 2)  \psi_i.$$  
 Then, 
\begin{equation*}\label{eq:case3-3}
\begin{aligned}\relax
[\phi_1 X_\alpha, \xi_{-1} X_{-\alpha}]  & = \xi H_\alpha + \frac{1}{4}\sum_{i = 2}^n m_i(m_i + 2) \theta_i \xi_{-i} X_{-\alpha} - \frac{1}{4} m_1(m_1 + 2) \psi_1 \xi_{-1} X_\alpha, \\
[\theta_i H_{\alpha}, \xi_{-i} X_{-\alpha}] & = -2\theta_i \xi_{-i} X_{-\alpha} - \xi H_{\alpha},\\
[\psi_i H_{\alpha}, \xi_{-i} X_{\alpha}] & = 2\psi_i \xi_{-i} X_{\alpha} - \xi H_{\alpha}.
\end{aligned}
\end{equation*}
It follows that 
\begin{multline*}
2 [\phi_1 X_\alpha, \xi_{-1} X_{-\alpha}]  
+ \frac{1}{4}\sum_{i = 2}^n m_i(m_i + 2) [\theta_i H_\alpha, \xi_{-i} X_{-\alpha}] 
 + \frac{1}{4}m_1(m_1 + 2) [\psi_1 H_\alpha, \xi_{-1} X_\alpha] \\
 = \left (2 -  \frac{1}{4}\sum_{i = 1}^n m_i(m_i + 2)\right )\xi H_\alpha \in \mathcal{L}_n.
\end{multline*}
Since $n \geq 2$ and since each $m_i$ is a positive, even integer,  $\frac{1}{4}\sum_{i = 1}^n m_i(m_i + 2) \geq 2n > 2$ and hence, $\xi H_\alpha\in \mathcal{L}_n$.

Now, let $H$ be an arbitrary element of $\frakh$. Decompose $H = H_1 + H_2$, where $H_1:= \frac{1}{2}\alpha_1(H) H_{\alpha_1}$ and $H_2  := H - \frac{1}{2}\alpha_1(H) H_{\alpha_1}$.  
It suffices to show that $\xi H_2\in \mathcal{L}_n$. 
Note that $\alpha_1(H_2) = 0$ and $H_2 \xi_{-1} = 0$. 
Thus, 
$$
[\psi_1 H_2,\xi_{-1}X_{\alpha_1}]
=  \psi_1\xi_{-1}\alpha_1(H_2) X_{\alpha_1}  + (H_2\xi_{-1})\psi_1 X_{\alpha_1} - (X_{\alpha_1}\psi_1)\xi_{-1}H_2 = -\xi H_2,
$$ 
which completes the proof. 
\end{proof}

We have thus shown that $\xi H\in \mathcal{L}_n$ for all $H\in \frakh$. 
To complete the proof of Theorem~\ref{thm:liealgebra}, it remains to establish the following result:  

\begin{proposition}
    For any $\alpha\in \Delta$, $\xi X_\alpha\in \mathcal{L}_n$. 
\end{proposition}

\begin{proof}
First, consider the case where there is an index $i\in \{1,\ldots, n\}$ such that $\lambda_i \neq -\alpha$. 
Let $H\in \frakh$ be such that $(\lambda_i + \alpha)(H)\neq 0$. Then, 
$$
[\phi_i X_\alpha, \xi_{-i}H] = -(\alpha + \lambda_i)(H) \xi X_\alpha  + \sum_{j = 1, j\neq i}^n  (X_\alpha \phi_j) \xi_{-j} H.
$$
Since $(X_\alpha \phi_j) \xi_{-j} H\in \mathcal{L}_n$, we have that $\xi X_\alpha\in \mathcal{L}_n$. 

Now, consider the case where $\lambda_i = -\alpha$ for all $i = 1,\ldots, n$.
We appeal to Lemma~\ref{lem:sl2rep} to obtain an element $\psi_1\in \Phi_0$ such that $X_{-\alpha} \psi_1 = \phi_1$ (since $\phi_1 \in \Phi_{-\alpha}$, the decomposition $\phi_1 = \phi'_1 + \phi''_1$ is such that $\phi''_1 = 0$). 
For each $i = 2,\ldots, n$, let $\eta_{-i}:= \psi_1 \prod_{j = 2,j\neq i}^n \phi_j$. Then, 
\begin{equation}\label{eq:laststep}
[\psi_1 X_{\alpha}, \xi_{-1}  X_{-\alpha}] = \psi_1 \xi_{-1} H_\alpha + \sum_{i = 2}^n (X_\alpha \phi_i) \eta_{-i} X_{-\alpha} - \xi X_\alpha \in \mathcal{L}_n.  
\end{equation}
The first term on the right hand side of~\eqref{eq:laststep} belongs to $\mathcal{L}_n$. For the second term, note that $X_\alpha \phi_i\in \Phi_0$ and $-\alpha \neq 0$, so the arguments in the first paragraph of the proof imply that $(X_\alpha \phi_i) \eta_{-i} X_{-\alpha}\in \mathcal{L}_n$. We thus conclude that $\xi X_\alpha\in \mathcal{L}_n$. 
\end{proof}

\section{Proofs of the main results}
\label{sec:proof-main}

We return to the setting where $G\subseteq \operatorname{GL}(\mathbb{V})$
is a real, connected, semi-simple matrix group and $\Phi$ is the space spanned by the matrix coefficients of the standard representation of $G$ on $\mathbb{V}$. 

We first establish Theorem~\ref{thm:main}. 

\begin{proof}[Proof of Theorem~\ref{thm:main}]
Let $\Phi''$ be the subspace of constant functions in~$\Phi$. Since $\Phi$ is a finite-dimensional $\frakg$-module and $\frakg$ is semi-simple, complete reducibility yields a $\frakg$-invariant decomposition $\Phi=\Phi'\oplus\Phi''$. Then, $\Phi'$ does not contain any nonzero constant function. Since $\mathcal{A}(\Phi')$ is unital and $\Phi''\subseteq\R\bfo$, we have $\mathcal{A}(\Phi')=\mathcal{A}(\Phi)$. 
By Lemma~\ref{lem:firsttwo}, Theorems~\ref{thm:degreeone} and~\ref{thm:liealgebra},  
$\mathcal{A}(\Phi')\frakg_L\subseteq \operatorname{Lie}\{W_1,W_2,W_3\}$. Also, since each $W_i$ is algebraic, we have that $\operatorname{Lie}\{W_1,W_2,W_3\} \subseteq \mathfrak{X}_{\mathrm{alg}}(G)$. 
It now remains to show that $\mathcal{A}(\Phi)\frakg_L = \mathfrak{X}_{\mathrm{alg}}(G)$. It is clear that any vector field in $\mathcal{A}(\Phi)\frakg_L$ is algebraic. Conversely, 
given a $V\in \mathfrak{X}_{\mathrm{alg}}(G)$, we show below that $V\in \mathcal{A}(\Phi)\frakg_L$. 
Since $G$ is a connected, semi-simple matrix group, $\det(x) = 1$ for all $x\in G$. Thus, $x^{-1} = \operatorname{adj}(x)$. It follows that $x^{-1} V(x)\in \frakg$ is algebraic.
Let $X_1,\ldots, X_m$ be a basis of $\frakg$. Then, there exist coefficients $a_1,\ldots, a_m\in \mathcal{A}(\Phi)$ such that 
$x^{-1} V(x) = \sum_{i = 1}^m a_i(x) X_i$,
which implies that $V = \sum_{i = 1}^m a_iL_{X_i}\in \mathcal{A}(\Phi)\frakg_L$. 
\end{proof}

The remainder of the section is dedicated to the proof of  Corollary~\ref{cor:smoothdiffapproximate}. 
We start with the following lemma:

\begin{lemma}\label{lem:smooth-Vfieldapproximate}
For any $k\in\N_0$, $\operatorname{Lie}\{W_1,W_2,W_3\}$ is dense in $\mathfrak X(G)$ with respect to the $\cok$-topology.
\end{lemma}

\begin{proof}
The set $\Phi \cup \{\bfo\}$ strongly separates points, i.e., $\bfo$ is nowhere zero and for any two distinct points $x, y\in G$, there exists a matrix coefficient $\phi\in \Phi$ such that $\phi(x)\neq \phi(y)$. Also, $\Phi$ separates tangent vectors. To wit, we fix a basis of $\mathbb{V}$ and let $c_{ij}(A)$ be the $ij$th entry of $A\in \operatorname{End}(\mathbb{V})$.   
Given a nonzero $A\in T_xG$, say $c_{ij}(A)\neq 0$, the matrix coefficient $c_{ij}: G\to \R$  satisfies $d_x c_{ij}(A) = c_{ij}(A)$. 
Thus, by Nachbin's Theorem~\cite[Theorem~1.2.1]{Llavona},   $\mathcal{A}(\Phi)$ is dense in $\mathrm{C}^k(G)$ with respect to the $\cok$-topology for any $k\in \N_0$. 
Next, let $X_1,\ldots, X_m$ be a basis of $\frakg$, so $L_{X_1},\ldots, L_{X_m}$ is a global frame on $G$. Then, for any $V\in \mathfrak{X}(G)$, there exist $a_1,\ldots, a_m\in \mathrm{C}^\infty(G)$ such that $V = \sum_{i = 1}^m a_i L_{X_i}$. 
It follows that
$\mathfrak{X}(G) = \operatorname{cl}_{\cok}(\mathcal{A}(\Phi)\frakg_L) = \operatorname{cl}_{\cok}(\operatorname{Lie}\{W_1,W_2,W_3\})$.
\end{proof}

Given a vector field $V\in \mathfrak{X}(G)$, the {\it vertical prolongation} of $V$, which we denote by $j^kV$ (following the notation of~\cite{krupka2001some}), is the smooth vector field on $J^k(G,G)$ such that 
$$
j^kV(j_x^k f) = \left. \frac{d}{dt} \right |_{t = 0} j^k_x(e^{tV}\circ f), 
$$
for any smooth map $f: G\to G$. It follows that 
\begin{equation}\label{eq:exponentialofprolongV}
e^{t j^kV}(j_x^kf)=j_x^k(e^{tV}\circ f).
\end{equation}
In particular, if $V$ is complete, then so is $j^k(V)$. 
Note that the map $j^k :\mathfrak{X}(G)\to \mathfrak{X}(J^k(G,G))$, defined by sending $V$ to $j^kV$, is linear and continuous, where $\mathfrak{X}(G)$ is equipped with the $\cok$-topology and $\mathfrak{X}(J^k(G,G))$ is equipped with the $\mathsf{co}^0$-topology (one can establish this fact by using the local coordinates of fiber charts~\cite[Section 3]{krupka2001some}). 
Also, note that $j^k$ is Lie-bracket preserving, i.e., $j^k[U, V] = [j^kU, j^kV]$ for any $U, V\in \mathfrak{X}(G)$.  

Let $\operatorname{Diff}_c(G)$ be the subgroup of $\operatorname{Diff}_0(G)$ that is isotopic to $\operatorname{id}_G$ via a compactly supported isotopy, and $\mathfrak{X}_c(G)$ be the space of compactly supported vector fields on $G$.  

The next two lemmas must be known in the literature. 
For completeness of presentation, we include relatively short proofs.  
We first have the following fact: 

\begin{lemma}\label{lem:smooth-compactflowapproximate}
Let $V\in\mathfrak X_c(G)$, $K\subseteq G$ be compact, $k\in \N_0$, and $\epsilon>0$. Then, there exists a $U\in\operatorname{Lie}\{W_1,W_2,W_3\}$ such that $e^{t U}$ is well defined on an open neighborhood of $K$ for all $t\in [0,1]$ and, moreover, 
$d_{k,K}(e^U, e^V)<\epsilon.$
\end{lemma}

\begin{proof}
Since $V$ is compactly supported, $V$ is complete and hence, so is $j^kV$.    
For convenience, let $\iota(x):=j_x^k\operatorname{id}_G$ for any $x\in G$.
Let $$P:=\left\{e^{t j^kV}(\iota(x)) \mid x\in K,\ 0\leq t\leq1\right\},$$
which is a compact subset of $J^k(G, G)$. 
Let $Q$ be a compact neighborhood of $P$ in $J^k(G, G)$. Without loss of generality, we assume that  $\epsilon > 0$ is sufficiently small such that if $z\in J^k(G, G)$ satisfies $\operatorname{dist}(z, P) < \epsilon$, then $z$ belongs to the relative interior of $Q$. Since $j^kV$ is smooth and since $Q$ is compact, there exists a constant $C > 0$ such that 
$$\|j^kV(w) - j^kV(z)\| \leq C\|w- z\|,\quad \mbox{for all } w, z\in Q.$$  
Next, we choose $\delta>0$ such that $\delta e^C<\epsilon$. 
By Lemma~\ref{lem:smooth-Vfieldapproximate} and by the fact that the map $j^k: V\mapsto j^kV$ is continuous, there exists a $U\in \operatorname{Lie}\{W_1,W_2,W_3\}$ such that 
$$
\max_{z\in Q} \|j^k U(z) - j^k V(z) \| < \delta.
$$
We show below that $U$ is a desired vector field. 

Given an arbitrary $x\in K$, let $v(t):= e^{t j^k V}(\iota(x))$ and $u(t):= e^{t j^k U}(\iota(x))$. Then, for any $t \in [0, 1]$ such that $u(s)\in Q$ for all $s\in [0,t]$,   
$$
\begin{aligned}
\|u(t) - v(t)\|
&\leq\int_0^t\|j^kU(u(s))-j^kV(v(s))\|ds\\
&\leq\int_0^t\|j^kU(u(s))-j^kV(u(s))\|ds +\int_0^t\|j^kV(u(s))-j^kV(v(s))\|ds\\
&\leq\delta t + C\int_0^t\|u(s) - v(s)\|ds.
\end{aligned}
$$
By Gronwall's inequality~\cite[Lemma~2.7]{TeschlODE}, we obtain that
$$
\|u(t) - v(t)\| \leq \delta t e^{Ct} \leq \delta e^C <\epsilon.
$$
Since $v(t)\in P$, the above inequality implies that $u(t)$ belongs to the relative interior of $Q$.  
By construction of $P$, $v(t)\in P$ for all $t\in [0,1]$. 
Then, using the standard first-exit arguments, we obtain that $u(t)$ belongs to the relative interior of $Q$ for all $t\in[0,1]$. 
Finally, note that $u(1) = j_x^k(e^U)$ and $v(1) = j_x^k(e^V)$ and hence, $\|j^k_x(e^U) - j^k_x(e^V)\| < \epsilon$. This holds for all $x\in K$, so $d_{k,K}(e^U,e^V) < \epsilon$. 
\end{proof}
 
We next have the following lemma:

\begin{lemma}\label{lem:finitewordsapproximation}
Let $U\in \operatorname{Lie}\{W_1,W_2,W_3\}$, $K\subseteq G$ be compact, $k\in \N_0$, and $\epsilon > 0$. Suppose that $e^{tU}$ 
is well defined on an open neighborhood of $K$ for all $t\in [0,1]$; then, there exists an element $f\in \mathcal{W}$ such that $d_{k,K}(e^{U}, f) < \epsilon$.  
\end{lemma}

\begin{proof}
Let $\iota(K):= \{\iota(x) \mid x \in K\}$. Then, $e^{t j^k U}$ is well defined on $\iota(K)$ for all $t\in [0,1]$. 
Note that each $j^k W_i$, for $i = 1,2,3$, is a complete vector field and that $j^kU\in \operatorname{Lie}\{j^kW_1, j^kW_2, j^kW_3\}$. Using the same arguments in~\cite{Varolin01, ForstnericStein},
there exist sequences $t_1,\ldots, t_m\in \R$ and $j^k W_{i_1},\ldots, j^k W_{i_m}$, with $i_j\in \{1,2,3\}$, such that 
$$\max_{z\in \iota(K)}\| e^{j^k U}(z) - e^{t_m j^k W_{i_m}}\circ \cdots\circ e^{t_1 j^k W_{i_1}}(z) \| < \epsilon.$$
We then let $f:= e^{t_m W_{i_m}}\circ \cdots \circ e^{t_1 W_{i_1}}$, and conclude the proof by noting that $j^k_x (e^U) = e^{j^k U}(\iota(x))$ and $e^{t_m j^k W_{i_m}}\circ \cdots\circ e^{t_1 j^k W_{i_1}}(\iota(x)) = j^k_xf$ for all $x\in K$.   
\end{proof}

With the lemmas above, we now prove Corollary~\ref{cor:smoothdiffapproximate}:

\begin{proof}[Proof of Corollary~\ref{cor:smoothdiffapproximate}]
    Let $g\in \operatorname{Diff}_0(G)$, $K\subseteq G $ be compact,  and $\epsilon >0$. 
    Let $h: G\times [0,1]\to G$ be a smooth isotopy with $h(\cdot, 0) = \operatorname{id}_G$ and $h(\cdot, 1) = g$. Let $\mathcal{N}$ be an open neighborhood of $K$ in $G$ with compact closure, and $\rho: G\to [0,1]$ be a smooth, compactly supported function such that $\rho(x) = 1$  for any $x\in \cup_{t\in[0,1]}h(\mathcal{N}, t)$.
    Let $U(x,t)$ be the time-varying vector field that generates the isotopy~$h$, which satisfies    
    $U(h(x,t), t):=\frac{\partial}{\partial t} h(x, t)$.  
    Now, let $\tilde g$ be the flow generated by $\rho U$ at $t = 1$. It follows that $\tilde g\in\operatorname{Diff}_c(G)$ and  $\tilde g|_{\mathcal{N}} =g|_{\mathcal{N}}$. Hence $d_{k,K}(g,f)=d_{k,K}(\tilde g,f)$ for every $f\in\operatorname{Diff}(G)$. 
    
    We exhibit below an element $f\in \mathcal{W}$ such that $d_{k,K}(\tilde g,f) < \epsilon$. It is known (see, e.g.,~\cite[Theorem~1]{HRTDiff} and the paragraph right after)
    that the group $\operatorname{Diff}_c(G)$ is simple. 
    Because the group generated by $e^V$ for $V\in \mathfrak{X}_c(G)$ is a normal subgroup of $\operatorname{Diff}_c(G)$,  
    there exist an $m\in \N$ and $V_1,\ldots, V_m\in \mathfrak{X}_c(G)$ such that 
    $\tilde g = e^{V_m} \circ \cdots \circ e^{V_1}$. 
    Let $K_1, \ldots, K_m$ be compact subsets of $G$ such that $e^{V_j}(K_{j-1})\subseteq \rint K_j$ for all $j = 1,\ldots, m$, where $K_0:= K$. 
    By Lemmas~\ref{lem:smooth-compactflowapproximate} and~\ref{lem:finitewordsapproximation}, for any $j$ and any $\epsilon_j > 0$, there exists an element $f_j\in \mathcal{W}$ such that $f_j(K_{j-1})\subseteq \rint K_j$ and, moreover, $d_{k, K_{j-1}}(e^{V_j},f_j) < \epsilon_j$. 
    Let $f:=f_m\circ \cdots \circ f_1$. 
    Since the composition map that sends $(h_1, h_2)\in \mathrm{Diff}(G)\times \mathrm{Diff}(G)$ to $h_2\circ h_1\in \mathrm{Diff}(G)$ is continuous with respect to the $\cok$-topology, we can choose $\epsilon_j$'s sufficiently small such that $d_{k,K}(\tilde g, f) < \epsilon$. 
\end{proof}

\printbibliography

\end{document}